\documentclass[11pt,reqno]{amsart}

\usepackage[english]{babel}
\usepackage{amsmath,amssymb,amsthm}
\usepackage{bbm}
\usepackage{calrsfs}
\usepackage{mathrsfs}

\usepackage[pdftex]{color}
\usepackage[dvipsnames]{xcolor}
\usepackage[bookmarks=true,hyperindex,pdftex,colorlinks,
citecolor=Black,linkcolor=Black, urlcolor=Black
]{hyperref}

\usepackage{ulem} 

\numberwithin{equation}{section}

\theoremstyle{plain}
\newtheorem{theorem}{Theorem}[section]
\newtheorem{theo}[theorem]{Theorem}
\newtheorem{cor}[theorem]{Corollary}
\newtheorem{prop}[theorem]{Proposition}

\theoremstyle{definition}

 \DeclareMathOperator{\supp}{supp}

 \DeclareMathOperator{\ran}{ran}
 \DeclareMathOperator{\diam}{diam}

\newcommand{\R}{\mathbb{R}}
\newcommand{\N}{\mathbb{N}}

\newcommand{\calK}{\mathscr{K}}
\newcommand{\calL}{\mathscr{L}}

\newcommand{\dopu}{{:}\allowbreak\ }
\newcommand{\eps}{\varepsilon}
\newcommand{\Id}{\mathrm{Id}}
\newcommand{\iy}{\infty}

\newcommand{\Lip}{\mathrm{Lip}}
\newcommand{\lip}{\mathrm{lip}}
\newcommand{\LipM}{\mathrm{Lip}_0(M)}
\newcommand{\lipM}{\mathrm{lip}_0(M)}

\def\bea{\begin{align*}}
\def\beq{\begin{equation}}
\def\eeq{\end{equation}}
\newcommand{\begsta}{\begin{statements}}
\def\endsta{\end{statements}}

\newcounter{abc}

\newenvironment{statements}%
{\setcounter{abc}{0}
\begin{list}%
{{\rm (\alph{abc})}}
{\usecounter{abc}
\parsep=0pt plus 1pt
\topsep=1pt plus 2pt minus 1pt
\itemsep=1pt plus 2pt minus 1pt
\leftmargin=3\baselineskip
\labelsep=.6\baselineskip
\labelwidth=2.4\baselineskip
\rightmargin 0pt}%
}%
{\end{list}}

\newcommand{\norm}{\mbox{$\|\:.\:\|$}}

\newcommand{\rest}[2]{#1\raisebox{-0.3ex}{\mbox{$\mid_{#2}$}}} 

\makeatletter
\@namedef{subjclassname@2020}{%
  \textup{2020} Mathematics Subject Classification}
\makeatother

\begin{document}

\title{The Perfekt theory of $M$-ideals}

\author{Dirk Werner}

\address{Department of Mathematics \\ Freie Universit\"at Berlin \\
Arnimallee~6 \\ D-14195~Berlin \\ Germany\newline
\href{http://orcid.org/0000-0003-0386-9652}{ORCID: \texttt{0000-0003-0386-9652}}
}
\email{werner@math.fu-berlin.de}



\begin{abstract}
We revisit some ideas of K.-M.~Perfekt who has provided an elegant framework to detect the biduality between function or sequence spaces defined in terms of some $o$- resp.\ $O$-condition. We present new proofs under somewhat weaker assumptions than before and apply the result to Lipschitz spaces. 
\end{abstract}

\date{6.5.21; version v1}

\dedicatory{Dedicated to Ehrhard Behrends, teacher, mentor and friend, \\
on the occasion of his 75th birthday.}

\maketitle
\thispagestyle{empty} 


\section{Introduction}

It has long been known that a number of pairs $(E_0, E)$ of function or sequence spaces defined in terms of a ``little $o$''- resp.\ ``big $O$''-condition provide examples of spaces in biduality, i.e., $E_0^{**} =E$ with the identical inclusion $E_0 \hookrightarrow E$ corresponding to the canonical embedding $\beta_{E_0}$ of $E_0$ into its bidual. In addition it turns out that often $E_0$ is an $M$-ideal in $E$ (the definition will be recalled shortly); see for example \cite{W92}. The best known and simplest example of this kind is the pair $(c_0, \ell_\infty)$; another example is the pair $(B_0, B)$ consisting of the ``little'' Bloch space and the usual Bloch space. 

In \cite{P1} and \cite{P2}, K.-M. Perfekt provided an elegant general framework to accommodate many examples of this phenomenon including the pair (VMO, BMO). Other examples and applications can be found in \cite{OGPSS}. In this note we shall revisit his construction, detailed in Section~\ref{sec3}, and will give new proofs under somewhat less restrictive assumptions. We then apply this framework to Lipschitz spaces over general compact pointed metric spaces, which was left out in \cite{P1} where only compact subsets of $\R^n$ were considered. 

Let us finish this section by recalling the notion of an $M$-ideal introduced by E.M.~Alfsen and E.G.~Effros in their seminal paper \cite{AE}. Let $E$ be a Banach space and $E_0$ a closed subspace. Then $E_0$ is called an $M$-ideal if there is a projection $P\dopu E^*\to E^*$ with $\ker P= E_0^\bot$, the annihilator of $E_0$ in $E^*$, such that 
\beq\label{eq1.1}
\|x^*\| = \|Px^*\| + \|x^*-Px^*\| \qquad \text{for all }x^*\in E^*.
\eeq
Detailed information on $M$-ideals can be found in E.~Behrends's monograph \cite{Beh} and in~\cite{HWW}. 

A very important special situation is when $E=E_0^{**}$; in this case $E_0$ is called $M$-embedded; see Chapter~III in \cite{HWW}. One can show that $E_0$ is $M$-embedded if and only if it is the restriction map $P\dopu x^{***} \mapsto \rest{x^{***}}{E_0}$ that satisfies (\ref{eq1.1}). 

The key result of part~I of the Alfsen-Effros paper is a characterisation of the $M$-ideal property in purely geometric terms, by means of an intersection property of balls. We shall employ the following version, originally due to \AA.~Lima; cf.\ \cite[Th.~I.2.2]{HWW}. The closed unit ball of a Banach space $E$ is denoted by~$B_E$.

\begin{theo}\label{1.1}
$E_0$ is an $M$-ideal in $E$ if and only if the following 3-ball property holds: For all $x\in B_E$, $y_1,y_2,y_3\in B_{E_0}$ and $\eps>0$ there is some $y\in E_0$ such that 
\beq\label{eq1.2}
\|x+y_i-y\|\le 1+\eps \qquad (i=1,2,3).
\eeq
\end{theo}

This is a very versatile tool to prove the $M$-ideal property since no prior information on the dual space is involved.


\section{$M$-ideals in subspaces of $C^b(L,Y)$.}
\label{sec2}

Perfekt's approach naturally leads to subspaces of $C^b(L,Y)$, the space of bounded continuous functions on a locally compact Hausdorff space $L$ with values in a Banach space~$Y$. 

We have the following result; $C_0(L,Y)$ stands for the space of continuous functions vanishing at infinity, i.e., a continuous function $f$ is in $C_0(L,Y)$ if and only if $\{t\in L\dopu \|f(t)\|_Y\ge\eps\}$ is compact for each $\eps>0$.

\begin{theo}\label{2.1}
Let $L$ be a locally compact Hausdorff space, $Y$ be a Banach space, let $E\subset C^b(L,Y)$ be a closed subspace and $E_0=E\cap C_0(L,Y)$. Assume that $B_{E_0}$ is dense in $B_E$ for the topology of uniform convergence on compact subsets of~$L$. Then $E_0$ is an $M$-ideal in~$E$.
\end{theo}

\begin{proof}
We shall verify the 3-ball property from Theorem~\ref{1.1}. So let $\eps>0$, $f\in B_E$ and $g_1, g_2, g_3\in B_{E_0}$. In the following we shall use the notation $\|h\|_S = \sup_{t\in S} \|h(t)\|_Y$ for a function on $L$ and $S^c=L\setminus S$ for the complement of~$S$. 

In the first step take a compact subset $K_0\subset L$ such that all $\|g_i\|_{K_0^c}\le \eps$. By the density assumption there is some $h_1\in B_{E_0}$ such that $\|f-h_1\|_{K_0}\le\eps$. Pick a compact set $K_1\supset K_0$ such that $\|h_1\|_{K_1^c}\le\eps$; obviously  $\|g_i\|_{K_1^c}\le \eps$ as well. 

In the nest step pick some function $h_2\in B_{E_0}$ such that $\|f-h_2\|_{K_1}\le\eps$ and 
a compact set $K_2\supset K_1$ such that $\|h_2\|_{K_2^c}\le\eps$; clearly, still $\|h_1\|_{K_2^c}\le\eps$. Inductively, one can find functions $h_j\in B_{E_0}$ and compact sets $K_0\subset K_1 \subset K_2 \subset \dots \subset L$ such that 
\beq\label{eqA}
\|f-h_j\|_{K_u} \le \eps \qquad \text{for }j>u
\eeq
and 
\beq\label{eqB}
\|h_j\|_{K_u^c} \le \eps \qquad \text{for }j\le u.
\eeq

Let $r>1/\eps$ and 
$$
g= \frac 1r \sum_{j=1}^r h_j \in B_{E_0}.
$$
We shall verify that 
\beq\label{eq2.3}
\|f+g_i-g\|\le 1+3\eps \qquad (i=1,2,3),
\eeq
which implies the 3-ball property.

Indeed, if $t\in K_0$, then $\|f(t)-h_j(t)\|_Y \le \eps$ for all~$j$ and $\|g_i(t)\|_Y\le 1$ for all~$i$, hence
$$
\|f(t)+g_i(t)-g(t)\|_Y \le \|g_i(t)\|_Y + \|f(t)-g(t)\|_Y \le 1+\eps.
$$
Next, suppose $t\in K_u\setminus K_{u-1}$ for some $u\in\{1,\dots,r-1\}$. Then we have
$\|f(t)-h_j(t)\|_Y \le 1+\eps$ for $j<u$ by (\ref{eqB}) and the triangle inequality and
$\|f(t)-h_j(t)\|_Y \le \eps$ for $j >u$ by (\ref{eqA}), and trivially 
$\|f(t)-h_u(t)\|_Y \le 2$. This shows for such $t$
\begin{align*}
\|f(t) + g_i(t) -g(t)\|_Y 
&\le 
\|f(t) -g(t)\|_Y + \| g_i(t) \|_Y \\
& \le 
\frac1r \bigl( (u-1)(1+\eps) +2 + (r-u)\eps \bigr) + \eps  \\
& = 
\frac {u+1}r + \Bigl( \frac {r-1}r +1 \Bigr)\eps \le 1+2\eps
\end{align*}
Finally, if $t\notin K_{r-1}$, then $\|h_j(t)\|_Y \le \eps$ for $j\le r-1$ by (\ref{eqB}) and $\|h_r(t)\|_Y \le 1$ so that 
\begin{align*}
\|f(t) + g_i(t) -g(t)\|_Y 
&\le 
\|f(t)\|_Y + \| g_i(t)\|_Y +\|g(t)\|_Y \\
& \le 
1 + \eps + \frac1r \bigl( (r-1)\eps + 1  \bigr) 
\le 1+3 \eps  .
\end{align*}

Altogether we have proved (\ref{eq2.3}).
\end{proof}


\section{The Perfekt construction}
\label{sec3}

We now recall the setup of Perfekt's approach; actually, we have removed some unnecessary restrictions. Let $X$ be a reflexive space and $Y$ be any Banach space. Consider a subset $\calL \subset L(X,Y)$ (the space of bounded linear operators from $X$ into $Y$) and equip it with a locally compact Hausdorff topology $\tau$ that is finer than the strong operator topology sot. Hence, for each $x\in X$, the mapping $\hat x\dopu \calL \to Y$, $\hat x(T)=Tx$, is continuous on $(\calL, \tau)$. Now define the vector subspace 
$$
E= \Bigl\{x\in X\dopu \sup_{T\in \calL} \|Tx\|_Y < \infty \Bigr\}
$$
of $X$. By definition, $\hat x\in C^b(\calL, Y)$ for $x\in E$. We further assume that $\calL$ is rich enough to make $x\mapsto \hat x$ injective and consequently $x\mapsto \|\hat x\|_\infty = \sup_{T\in\calL} \|Tx\|_Y$ a norm on $E$. In this situation $\|x\|_\infty := \|\hat x\|_\infty$ is a norm on $E$  which makes $E$ isometric to a subspace of $C^b(\calL, Y)$; henceforth we shall consider $E\subset C^b(\calL,Y)$ in a canonical way. We also assume that
$E$ is closed in $C^b(\calL, Y)$ so that both $E$ and $E_0:= E\cap C_0(\calL, Y)$ are Banach spaces. Then  the canonical inclusion mapping $x\mapsto x$ from $(E, \norm_\infty)$ to $(X, \norm_X)$ is continuous by the closed graph theorem.
These assumptions will be in place throughout the whole section. 

We finally consider the following crucial density assumption.
\begsta
\item[(A)]
$B_{E_0}$ is dense in $B_E$ for the topology generated by the norm $\norm_X$ of~$X$.
\endsta

Under these assumptions we will now give a new proof of Perfekt's biduality theorem \cite[Th.~2.2]{P1}. At several points, the argument below is the same as in \cite{P1}, but it differs at a decisive juncture so that we can dispense with some assumptions in \cite{P1}. 
Instead of using Singer's theorem representing  the dual of $C_0(\calL,Y)$ by vector measures our argument relies on the $C_0(\calL)$-module structure of that space. 

\begin{theo}\label{3.1}
The space $E$ is canonically isometric to $E_0^{**} $ provided assumption~\mbox{\rm{(A)}} is valid. 
\end{theo}

Before proceeding to the proof let us explain which isomorphism is meant in the theorem and what makes it canonical. Let us introduce, as in \cite{P1}, the operators $i_0\dopu E_0\to X$, $i_0(x)=x$; $J\dopu X^*\to E_0^*$, $J=i_0^*$ (i.e., $Jx^*= \rest{x^*}{E_0}$); $I\dopu E_0^{**}\to X$, $I=J^*$. Then the claim is that $\ran I=E$ and $I$ is an isometry for $\norm_\iy$. Note that 
$I\beta_{E_0} = \Id_{E_0\to E}$ (with $\beta_{E_0}$ the canonical map from $E_0$ into its bidual), which makes $I$ canonical.  

\begin{proof}
In the first step we prove that $\ran J$ is norm dense in $E_0^*$. Let $\ell^0\in E_0^*$ and consider a Hahn-Banach extension $\ell\in C_0(\calL, Y)^*$. Let $\varphi\in C_0(\calL)$ with compact support $\calK$ and $0\le \varphi\le 1$; then the functional 
$$
\ell_\varphi\in C_0(\calL, Y)^*, \quad \ell_\varphi(f)= \ell(\varphi f)
$$
is well defined. If $x\in E_0$ we have
$$
|\ell_\varphi(x)| = |\ell(\varphi x)| \le \|\ell\| \cdot \|\varphi x\|_\iy 
$$
and
$$
\sup_{T\in \calL} \|\varphi(T) Tx\|_Y = \sup_{T\in \calK} \|\varphi(T) Tx\|_Y \le 
\sup_{T\in \calK} \|T\| \cdot \|x\|_X.
$$
Now $\calK$ is $\tau$-compact and hence sot-compact; as a result it is pointwise bounded and so $\sup_{T\in \calK} \|T\|<\iy$ by the uniform boundedness principle. This shows that $\ell_\varphi$ is continuous on $E_0$ for the norm of $X$ and therefore the restriction of some $x^*\in X^*$ to~$E_0$; consequently $\rest{\ell_\varphi}{E_0}\in \ran J$.

Let us now prove:
\begsta
\item[$\bullet$] \em
For every $\eps>0$ there exists $\varphi\in C_0(\calL)$ with compact support, $0\le \varphi\le 1$, such that $\|\ell- \ell_\varphi\|_{C_0(\calL, Y)^*} \le\eps$.
\endsta 
Indeed, assume that for some $\eps_0>0$ there is no such $\varphi$.  Then for all such $\varphi$
$$
\|\ell_{1-\varphi}\|_{C_0(\calL, Y)^*} =
\|\ell- \ell_\varphi\|_{C_0(\calL, Y)^*} >\eps_0.
$$
Let us start with $\varphi_1=0$; so there is some $f_1\in C_0(\calL, Y)$, $\|f_1\|_\iy =1$, with $|\ell(({1-\varphi_1})f_1)|> \eps_0$; upon replacing $f_1$ with $-f_1$ if necessary we even have
$$
\ell(({1-\varphi_1})f_1)> \eps_0.
$$
Since the functions of compact support are dense in $C_0(\calL,Y)$ we may as well assume that $\supp f_1$ is compact. Next consider a function $\varphi_2\in C_0(\calL)$ with compact support, $0\le \varphi_2\le 1$, such that $\varphi_2(T)=1$ on $\supp f_1$. By our assumption there is some $f_2\in C_0(\calL, Y)$ of compact support, $\|f_2\|_\iy=1$, such that 
$$
\ell(({1-\varphi_2})f_2)> \eps_0.
$$
Inductively we find $\varphi_j\in C_0(\calL)$ between $0$ and $1$ and
$f_j\in C_0(\calL, Y)$ of norm~$1$, both with compact support, such that $\varphi_j(T)=1$ on $\supp f_1 \cup \dots \cup \supp f_{j-1}$ and
$$
\ell(({1-\varphi_j})f_j)> \eps_0.
$$
By construction, if $(1-\varphi_j)(T) f_j(T)\neq0$, then $(1-\varphi_i)(T) f_i(T) =0$ for all $i< j$; consequently for all $r\in \N$
$$
\Bigl\| \sum_{j=1}^r (1-\varphi_j) f_j \Bigr\| _\iy \le 1,
$$
but 
$$
\ell\Bigl( \sum_{j=1}^r (1-\varphi_j) f_j \Bigr) > r\eps_0;
$$
this is a contradiction if $r\ge \|\ell\|/\eps_0$. 
Thus, the above claim is proved. 

Since the estimate in the above claim is a fortiori true for the restrictions of the functionals to $E_0$, we get that, given $\eps>0$,
$$
\|\ell^0 -\rest{\ell_\varphi}{E_0}\| \le\eps.
$$
Together with the first part of the proof this implies that the range of $J$ is dense in~$E_0^*$. 

Now that we know that $J$ has dense range, it is clear that $I$ is injective. Let us show that $\ran I \subset E$. Let $e_0^{**}\in E_0^{**}$; we wish to prove that the element $Ie_0^{**}\in X$ satisfies
$$
\sup_{T\in \calL} \|T(Ie_0^{**})\|_Y<\infty.
$$
Without loss of generality we can assume $\|e_0^{**}\|=1$. Then there is a net $(e_\alpha)$ in $B_{E_0}$ such that $e_\alpha\to e_0^{**}$ for the topology $\sigma(E_0^{**}, E_0^*)$. Since $I=J^*$ is weak$^*$ continuous, it follows $Ie_\alpha \to Ie_0^{**}$ for the topology $\sigma(X,X^*)$. Since $I \beta_{e_0} e_\alpha  =e_\alpha$, we can conclude for each $T\in \calL$ and $y^*\in Y^*$
$$
(T^*y^*)e_\alpha \to (T^*y^*)(Ie_0^{**}),
$$
that is
$$
y^*(Te_\alpha) \to y^*(TIe_0^{**}).
$$
Now
$$
|y^*(Te_\alpha)| \le \|y^*\| \, \|Te_\alpha\|_Y \le \|y^*\| \, \|e_\alpha\|_\infty \le \|y^*\|
$$
and therefore 
$$
|y^*(TIe_0^{**})| \le \|y^*\|
$$
which shows
$$
\|T(Ie_0^{**})\|_Y = \sup_{\|y^*\|\le1}  |y^*(TIe_0^{**})| \le 1
$$
and thus $\|Ie_0^{**}\|_\iy \le 1$. This proves that $Ie_0^{**}\in E$ and that $I$ is a contraction as an operator from $E_0^{**}$ to $(E,\norm_\iy)$.

Finally consider the linear mapping $\Lambda\dopu E\to E_0^{**}$, $(\Lambda e)(\rest{x^*}{E_0}) = x^*(e)$, which is well defined since $\ran J$ is dense in~$E_0^*$. One has, given $e\in E$, 
$$
x^*(e)= (\Lambda e)(Jx^*) =(I\Lambda e)(x^*) \qquad \text{for all }x^*\in X^*,
$$
hence $e=I\Lambda e$; therefore $I$ is surjective and
$$
\|e\| = \|I\Lambda e\| \le \|\Lambda e\|.
$$

To complete the proof of the theorem it remains to show that $\Lambda$ is contractive, that is 
$$
x^*\in X^*,\ |x^*(e_0)|\le 1 \text{ for all }e_0\in B_{E_0} 
\quad \Rightarrow \quad
|x^*(e)|\le 1 \text{ for all }e\in B_{E}.
$$ 
It is here that the assumption (A) enters. Let $x^*\in X^*$ as above and $\|e\|_\iy=1$; pick a sequence $(e_n)$ in $B_{E_0}$ satisfying $\|e_n-e\|_X\to 0$, in particular $x^*(e_n)\to x^*(e)$. Consequently $|x^*(e)|\le 1$, as requested.
\end{proof}

\begin{cor}\label{3.2}
Under the assumptions of Theorem~\ref{3.1}, $E_0$ is an $M$-ideal in $E$ and hence an $M$-embedded space.
\end{cor}

\begin{proof}
We just have to verify the density condition of Theorem~\ref{2.1}, that is, $B_{E_0}$ is dense in $B_E$ for the topology of uniform convergence on compact subsets of~$\calL$. 

Let $e\in E$, $\|e\|_\iy=1$. By assumption (A) there are $(e_n)$ in $B_{E_0}$ such that $\|e_n-e\|_X\to 0$; as a result $\|Te_n-Te\|_Y \le \|T\|\,\|e_n-e\|_X \to 0$ for each $T\in \calL$. This says that the associated functions $\hat e_n$ on $\calL$ converge pointwise on~$\calL$. But we have already observed in the previous proof that a compact subset of $\calL$ is bounded for the norm of $L(X,Y)$ by the uniform boundedness principle. Hence the convergence is uniform on compact subsets of~$\calL$.
\end{proof}

There is a limit to this method of detecting $M$-embedded spaces since the simplest nonseparable $M$-embedded space, $c_0(\Gamma)$ for some uncountable~$\Gamma$, cannot be injected into a reflexive space. (This is a known result; here is a sketch of the proof: Suppose $j_0\dopu c_0(\Gamma)\to X$ is a continuous injection into a reflexive space. Then $j_0^*$ is weakly compact with pointwise dense range in $\ell_1(\Gamma)$, which has the Schur property; thus $j_0^*$ is compact and hence has norm-separable range. This means that $\ran j_0^* \subset \ell_1(\Gamma_0)$ for some countable $\Gamma_0\subset \Gamma$, a contradiction.)


\section{Lipschitz spaces}
\label{sec4}

Let $(M,d)$ be a metric space with a distinguished point $p\in M$ (a ``pointed'' metric space). We denote by $\LipM$ the Banach space of all real-valued Lipschitz functions on $M$ vanishing at $p$ with the Lipschitz constant as its norm:
$$
\|F\|_\Lip = \sup _{s\neq t} \frac {|F(s)-F(t)|}{d(s,t)}.
$$
We also consider the ``little'' Lipschitz space
$$
\lipM = \Bigl\{f\in \LipM\dopu \lim_{d(s,t)\to 0}  \frac {|f(s)-f(t)|}{d(s,t)} =0 \Bigr\}.
$$ 
This is a closed subspace of $\LipM$ which might however reduce to $\{0\}$, e.g., for $M=[0,1]$ with the Euclidean metric. 
For a H\"older metric on $[0,1]$, $d_\alpha(s,t)= |s-t|^\alpha$ where $0<\alpha<1$, the subspace $\lip_0([0,1], d_\alpha)$ is nontrivial, indeed $\Lip_0([0,1],d_\beta) \subset \lip_0([0,1], d_\alpha)$ for $\alpha<\beta\le1$.
The authoritative source about Lipschitz spaces is N.~Weaver's monograph~\cite{Wea}.

We shall now consider compact (pointed) metric spaces with a countable dense subset $P$ having the following property:

\begsta
\item[(B)]
$B_{\lipM}$ is dense in $B_{\LipM}$ for the topology of pointwise convergence on~$P$.
\endsta
Since $B_{\LipM}$ is equicontinuous this is the same as saying:

\begsta
\item[(B$'$)]
$B_{\lipM}$ is dense in $B_{\LipM}$ for the topology of uniform convergence on~$M$.
\endsta

\begin{prop}\label{4.1}
If $M$ is a compact pointed metric space with {\rm (B)}, then $\lipM$ is an $M$-ideal in $\LipM$.
\end{prop}

\begin{proof}
Let $\Delta_M=\{(t,t)\dopu t\in M\}$ be the diagonal in $M\times M$ and $L= (M\times M)\setminus \Delta_M$; this is a locally compact space. We further associate to a function $F$ on $M$ a new function $\Phi F$ on $L$ defined by 
$$
(\Phi F)(s,t)= \frac {|F(s)-F(t)|}{d(s,t)};
$$
this is the approach of K.~de Leeuw's classical paper \cite{deL}. Note that $\Phi$ is a linear isometry from $\LipM$ into $C^b(L)$ that takes $\lipM$ into $C_0(L)$. Let $\Lambda= \Phi(\LipM)$ and $\lambda= \Phi(\lipM)$; then $\lambda = \Lambda\cap C_0(L)$. 

To prove that $\lambda $ is an $M$-ideal in $\Lambda$ (and thus $\lipM$ is an $M$-ideal in $\LipM$) it is enough, by Theorem~\ref{2.1},  to show that $B_\lambda$ is dense in $B_\Lambda$ for the topology of uniform convergence on compact subsets of~$L$. If $K\subset L$ is compact, then $\inf\{d(s,t)\dopu (s,t)\in K\} =:\delta >0$. Let $F\in B_{\LipM}$; by (B) (or rather~(B$'$)) there are $f_n\in B_{\lipM}$ such that $\|f_n-F\|_\iy\to 0$; hence for $(s,t)\in K$
\bea
|(\Phi f_n) (s,t) - (\Phi F)(s,t)| &\le 
\frac1\delta |(f_n(s)-f_n(t)) - (F(s)-F(t))| \\
& \le
\frac1\delta ( |f_n(s)-F(s)| + |f_n(t)-F(t)| ) \\
& \le 
\frac2\delta \|f_n-F\|_\iy \to 0
\end{align*}
so that $\Phi f_n \to \Phi F$ uniformly on $K$.

This completes the proof of the proposition.
\end{proof}

We now come to the biduality theorem, originally due to N.~Weaver in \cite{Wea96} who was the first to handle $\lip_0$-spaces and even covered certain noncompact spaces. In \cite{P1} Perfekt considered the case of a H\"older metric on a compact subset  $M\subset \R^n$ and showed that $\lipM^{**} \cong \LipM$ by means of his method using a Besov space as the reflexive space~$X$. He asked for a proof along these lines for a general compact metric space; this will be accomplished in the proof of the next theorem. 

\begin{theo}\label{4.2}
Let $M$ be a compact pointed metric space satisfying {\rm (B)} above. Then $\lipM^{**} \cong \LipM$.
\end{theo}

\begin{proof}
We shall set up the Perfekt scenario as follows. Let $X$ be the weighted $\ell_2$-space 
$$
X= \Bigl\{ (x_n)\dopu \sum_{k=1}^\iy |x_k|^2 2^{-k} <\iy \Bigr\}
$$
with its canonical norm. Pick a dense sequence $(p_n)$ in $M$ and consider the functionals $\ell_{n,m}\in X^*$ (so $Y=\R$) defined by
$$
\ell_{n,m}(x)= \frac{x_n-x_m}{d(p_n, p_m)}\qquad (n\neq m)
$$
and equip $\calL= \{\ell_{n,m}\dopu n,m\in\N$, $n\neq m\}$ with the discrete topology. For $F\in \LipM$ let $x_F= (F(p_n))$, then $x_F\in X$ since $F$ is bounded. Further define $E= \{x_F\dopu F\in \LipM\}$. Note that
$$
\sup_{n\neq m} |\ell_{n,m}(x_F)| = \sup_{n\neq m} \Bigl| \frac{F(p_n)-F(p_m)}{d(p_n, p_m)} \Bigr| = \|F\|_\Lip;
$$
therefore $(E, \norm_\iy)$ is isometric to $\LipM$, and $E$ is a closed subspace of $C^b(\calL)$.

Let $E_0=E\cap C_0(\calL)$. We shall argue that $x_F\in E_0$ if and only if $F\in\lipM$. Indeed, write $\alpha \calL= \calL\cup \{\iy\}$ for the Alexandrov compactification of~$\calL$. Suppose $\lim_{\ell_{n,m}\to\iy} \ell_{n,m}(x_F)=0$. Then, given $\eps>0$, there is a finite set $K \subset (\N\times \N) \setminus \Delta_\N$ such that $|\ell_{n,m}(x_F)|<\eps$ for $(n,m)\notin K$. Let $\delta= \inf\{d(p_n, p_m)\dopu (n,m)\in K\} >0$. Therefore, if $d(p_n,p_m)<\delta$, then $(n,m)\notin K$ and as a result 
$$
 \Bigl| \frac{F(p_n)-F(p_m)}{d(p_n, p_m)} \Bigr| <\eps
$$
for these $(n,m)$, which implies $F\in \lipM$ by density of $\{p_1, p_2, \dots\}$. Conversely, if $F\in \lipM$, then, if $\eps>0$ and $\ell_{(n,m)} (F)\ge\eps$, we have $d(p_n,p_m)\ge\delta $ for some $\delta>0$. In a compact space there can only be finitely many $\delta$-separated points; this proves that $\lim_{\ell_{n,m}\to \infty} \ell_{n,m}(x_F)=0$.

To conclude the proof of the theorem it is only left to verify condition~(A) from Section~\ref{sec3} and to apply Theorem~\ref{3.1}. By (B$'$) there are, given $F\in \LipM$ with $\|F\|_\Lip=1$, functions $f_n\in B_{\lipM}$ such that $\|f_n-F\|_\iy\to0$. Then
\bea
\|x_{f_n}-x_F\|_X^2 
& = 
\sum_{k=1}^\iy |f_n(p_k)-F(p_k)|^2 2^{-k} \\
&\le
\sum_{k=1}^\iy \|f_n-F\|_\iy ^2 2^{-k} = \|f_n-F\|_\iy ^2 \to 0,
\end{align*}
as required.
\end{proof}

Let $(M,d)$ be a metric space. For $0<\alpha<1$, $d^\alpha$ is a metric as well; we write $M^\alpha$ to indicate that $M$ is equipped with the metric~$d^\alpha$. (Sometimes $M^\alpha$ is called a snow-flaked version of~$M$.) The corresponding Lipschitz spaces are also called Lipschitz-H\"older spaces, in accordance with the classical notation in~$\R^n$.  We shall prove that the compact metric spaces $M^\alpha$ satisfy the condition~(B). 

\begin{prop}\label{4.3}
Let $(M,d)$ be a compact pointed metric space and $0<\alpha<1$. Then $M^\alpha$ satisfies condition~{\rm(B)}.
\end{prop}

\begin{proof}
Let $P=\{p_1,p_2,\dots\}$ be a countable dense subset of $M$ with $p_1=p$, the base point of~$M$, and let $P_n=\{p_1,\dots,p_n\}$. Let $F\in \Lip_0(M^\alpha)$ with $\|F\|_{\Lip_0(M^\alpha)}=1$. Define a function $g_n$ on $P_n$ by $g_n(p_j)= F(p_j)$. Clearly $\|g_n\|_{\Lip_0(P_n^\alpha)}\le 1$. Choose $\beta_n\in (\alpha, 1)$ such that $\|g_n\|_{\Lip_0(P_n^{\beta_n})}\le 1 + \frac1n$ and $(\diam M)^{\beta_n-\alpha}\le 1+\frac1n$; the former is possible since $P_n$ is finite. Now apply the McShane extension theorem \cite[Th.~1.33]{Wea} and extend $g_n$ to a function $G_n\dopu M\to \R$ having the  same Lipschitz constant as $g_n$ for the metric~$d^{\beta_n}$. Finally let $f_n=G_n/(1+\frac1n)^2$.
Then $f_n\in \Lip_0(M^{\beta_n})$ and thus $f_n\in \lip_0(M^\alpha)$. Moreover, $f_n\to F$  pointwise on $P$ by construction and 
\begin{align*}
\|f_n\|_{\Lip_0(M^\alpha)} 
&\le \|f_n\|_{\Lip_0(M^{\beta_n})} (\diam M)^{\beta_n-\alpha} \\
&\le \frac{ \|G_n\|_{\Lip_0(M^{\beta_n})} }{ (1+\frac1n)^2 } \Bigl( 1+\frac1n \Bigr) \le 1.
\end{align*}
This proves condition (B) for $d^\alpha$.
\end{proof}

\begin{cor}\label{4.4}
Let $(M,d)$ be a compact pointed metric space and ${0<\alpha<1}$. Then $(\lip_0(M^\alpha))^{**}$ is canonically isometric to $\Lip_0(M^\alpha)$, and $\lip_0(M^\alpha)$ is an $M$-ideal in its bidual $\Lip_0(M^\alpha)$.
\end{cor}

\begin{proof}
This follows from Proposition~\ref{4.1}, Theorem~\ref{4.2} and Proposition~\ref{4.3}.
\end{proof}

It is a remarkable fact, proved by N.~Kalton \cite[Th.~6.6]{Kal}, that, for a compact metric space, $\lipM$ is always $M$-embedded since it embeds almost isometrically into~$c_0$. I do not know whether $\lipM$ is always an $M$-ideal in $\LipM$; this trivially holds if $\lipM$ is trivial.

The biduality theorem of Corollary~\ref{4.4} is originally due to N.~Weaver
\cite{Wea96}; 
in fact he gave a number of conditions that are equivalent to the validity of the biduality theorem on a compact metric space. Let us add that condition (B) is actually equivalent to $\LipM$ being the bidual of $\lipM$ under the canonical duality. Indeed, if $\lipM$ and $\LipM$ are in canonical biduality, then by Goldstine's theorem $B_{\lipM}$ is weak$^*$ dense in $B_{\LipM}$, and checking this on point evaluations $F\mapsto F(t)$, which span the predual $\mathscr{F}(M)$ of $\LipM$ (known as the Lipschitz free space, Arens-Eells space or transportation cost space), shows that (B) holds.


\end{document}